\documentclass[11]{amsart}

\usepackage{amsfonts}
\usepackage{amscd}
\usepackage{amsmath, mathrsfs, amssymb}
\usepackage{amsthm}
\usepackage{setspace}
\usepackage{hyperref}
\usepackage{color}
\usepackage{epsfig}
\usepackage{here}
\usepackage{graphicx}
\usepackage[all]{xy}
\usepackage{psfrag}
\usepackage{graphicx,transparent}
\usepackage{enumerate}
\usepackage{caption}

\theoremstyle{plain}
\newtheorem{theorem}{Theorem}[section]

\newtheorem{proposition}[theorem]{Proposition}

\theoremstyle{definition}
\newtheorem{remark}[theorem]{Remark}

\newtheorem{example}[theorem]{Example}

\newcommand{\MM}{\mathcal M}

\newcommand{\BM}{\overline{\mathcal M}}

\newcommand{\PP}{\mathcal P}

\newcommand{\CC}{\mathcal C}
\newcommand{\calC}{\mathcal C}

\newcommand{\OO}{\mathcal O}

\newcommand{\RR}{\mathbb R}

\newcommand{\HH}{\mathcal H}

\newcommand{\hyp}{\operatorname{hyp}}

\newcommand{\Pic}{\operatorname{Pic}}

\newcommand{\bbQ}{\mathbb Q}

\newcommand{\bbZ}{\mathbb Z}

\newcommand{\SL}{\operatorname{SL}}

\begin{document}

\makeatletter
	\@namedef{subjclassname@2010}{%
	\textup{2010} Mathematics Subject Classification}
	\makeatother

\title{Tautological ring of strata of differentials}

\author{Dawei Chen}
\address{Department of Mathematics, Boston College, Chestnut Hill, MA 02467}
\email{dawei.chen@bc.edu}


\date{\today}

\thanks{The author is partially supported by NSF CAREER Award DMS-1350396.}

\begin{abstract}
Strata of $k$-differentials on smooth curves parameterize sections of the $k$-th power of the canonical bundle with prescribed orders of zeros and poles. Define the tautological ring of the projectivized strata using the $\kappa$ and $\psi$ classes of moduli spaces of pointed smooth curves along with the tautological class $\eta$ of the Hodge bundle. We show that if there is no pole of order $k$, then the tautological ring is generated by $\eta$ only, and otherwise it is generated by the $\psi$ classes corresponding to the poles of order $k$. 
\end{abstract}

\maketitle

\section{introduction}
\label{sec:intro}

For a positive integer $k$, let $\mu = (m_1, \ldots, m_n)$ be an integral (ordered) partition of $k(2g-2)$, i.e., $m_i \in \bbZ$ and $\sum_{i=1}^n m_i = k(2g-2)$. The stratum of $k$-differentials $\HH^k(\mu)$ parameterizes $(C, \xi, p_1, \ldots, p_n)$, where $C$ is a smooth algebraic curve of genus $g$ and $\xi$ is a (possibly meromorphic) section of $K_C^{\otimes k}$ such that $(\xi)_0 - (\xi)_\infty = \sum_{i=1}^n m_i p_i$ for distinct (ordered) points $p_1, \ldots, p_n \in C$. If we consider differentials up to scale, then the corresponding stratum of $k$-canonical divisors $\PP^k(\mu)$ parameterizes the underlying divisors $ \sum_{i=1}^n m_i p_i$, and $\PP^k(\mu)$ is the projectivization of $\HH^k(\mu)$. We also write $\HH(\mu)$ and $\PP(\mu)$ for the case $k=1$. 

Abelian and quadratic differentials, i.e., the cases $k=1$ and $k=2$ respectively, have broad connections to flat geometry, billiard dynamics, and Teichm\"uller theory. Recently their algebraic properties (also for general $k$-differentials) have been investigated greatly, which has produced fascinating results in the study of Teichm\"uller dynamics and moduli of curves. We refer to \cite{ZorichFlat, WrightLectures, Bootcamp, EskinMirzakhani, EskinMirzakhaniMohammadi, Filip, FarkasPandharipande, BCGGM1, BCGGM2} for related topics as well as some recent advances. 

Let $\MM_{g,n}$ be the moduli space of smooth genus $g$ curves with $n$ marked points. For a signature $\mu = (m_1, \ldots, m_n)$ with $m_1, \ldots, m_r > 0$ and $m_{r+1}, \ldots, m_n \leq 0$, define the $k$-th Hodge bundle 
$\HH^k(\tilde{\mu})$ over $\MM_{g,n}$ (twisted by the nonpositive part $\tilde{\mu} = (m_{r+1}, \ldots, m_n)$ of $\mu$) as 
$\pi_{*} (\omega^{\otimes k} (-m_{r+1}p_{r+1} - \cdots - m_n p_n))$, where $\pi: \CC \to \MM_{g,n}$ is the universal curve, $\omega$ is the relative dualizing line bundle of $\pi$ and  
$p_{r+1}, \ldots, p_n$ are the last $n-r$ marked points. Then $\HH^k(\tilde{\mu})$ contains $\HH^k(\mu)$ as a subvariety. 
Denote by $\eta$ the tautological $\OO(-1)$ line bundle class of the projective bundle $\PP^k(\tilde{\mu})$ as well as its restriction to each projectivized stratum $\PP^k(\mu)\subset \PP^k(\tilde{\mu})$. 

Recall that the tautological ring $R^*(\MM_{g,n})$ is the subring of the Chow (or cohomology) ring of $\MM_{g,n}$ (in rational coefficients) generated by the $\kappa$ classes (or the $\lambda$ classes) and the $\psi$ classes, where the $\kappa$ are the Miller-Morita-Mumford classes,  the $\lambda$ come from the Chern classes of the Hodge bundle over $\MM_g$, and the $\psi$ are the cotangent line bundle classes associated to the marked points (see e.g., \cite{Faber, FaberPandharipande}). We still use $\kappa$, $\lambda$ and $\psi$ to denote the pullbacks of these classes to $\PP^k(\mu)$. 

Since the Chow ring of a projective bundle is generated by cycle classes from the base along with the $\OO(-1)$ class of the bundle, it is natural to define the tautological ring $R^*(\PP^k(\mu))$ of the strata as $R^*(\MM_{g,n})[\eta]$, i.e., $R^*(\PP^k(\mu))$ is generated by the $\kappa$ (or $\lambda$), $\psi$ and $\eta$ classes (see also \cite{Sauvaget} for a similar definition of the tautological ring for the compactified strata over $\BM_{g,n}$). For special $\mu$, the strata $\PP^k(\mu)$ can be disconnected (see \cite{KontsevichZorich, Lanneau, Boissy}). In that case we define similarly the tautological ring 
for each connected component of $\PP^k(\mu)$. Our main result gives a surprisingly simple description of the structure of $R^*(\PP^k(\mu))$ as follows. 

\begin{theorem}
\label{thm:tauto}
For $\mu = (m_1, \ldots, m_n)$, if $m_i \neq -k$ for all $i$, then $R^*(\PP^k(\mu))$ is generated by $\eta$ only. Otherwise $R^*(\PP^k(\mu))$ is generated by the $\psi_i$ with $m_i = -k$. 
\end{theorem}

We prove a more detailed version of Theorem~\ref{thm:tauto} (see Proposition~\ref{prop:tauto}). Moreover, we provide some examples for which the generators in Theorem~\ref{thm:tauto} cannot be further reduced. 

\subsection*{Acknowledgements} Part of the work was carried out when the author visited Fudan University in Summer 2017. The author sincerely thanks Meng Chen for the invitation and hospitality. The author also thanks Qile Chen, Sam Grushevsky, Felix Janda and Zhiyuan Li for helpful discussions on related topics. 

\section{Tautological ring of the strata}
\label{sec:tauto}

In this section we prove Theorem~\ref{thm:tauto}. We first prove the following result. 

\begin{proposition}
\label{prop:tauto}
For $\mu = (m_1, \ldots, m_n)$, the tautological classes on $\PP^k(\mu)$ satisfy the following relations: 
\begin{itemize}
\item[(i)] $\eta = (m_i + k) \psi_i$ for all $i$. 
\item[(ii)] $\kappa_j = - \sum_{i=1}^n \psi_i^{j}$ for all $j$. 
\end{itemize}
\end{proposition}

\begin{proof}
Let $\pi: \calC\to B$ be a flat family of smooth curves with $n$ distinct sections $S_1, \ldots, S_n$ such that $\sum_{i=1}^n m_i S_i$ restricted to each fiber is a (possibly meromorphic) $k$-canonical divisor, i.e., $B$ parameterizes $k$-canonical divisors of type $\mu$. It suffices to verify the relations restricted to $B$. Denote by $\omega$ the relative dualizing line bundle class of $\pi$. Since $\eta$ represents the line bundle $\OO(-1)$ of generating $k$-differentials of the family, comparing $\pi^{*}\eta$ with $\omega^{\otimes k}$, they are fiberwise isomorphic but the former has a zero or pole along each section $S_i$ with multiplicity $m_i$. In other words, the following relation of divisor classes holds on $\calC$: 
\begin{eqnarray}
\label{eq:tauto}
\pi^{*}\eta = k \omega -\sum_{i=1}^n m_i S_i 
\end{eqnarray}
(see also \cite[Section 3.4]{EKZ} for the case $k=1$).
Note that 
$$ S_i \cdot S_j = 0, \quad i \neq j, $$
$$ \pi_{*} (S_i^2) = - \psi_i, $$ 
$$ \pi_{*} (\omega \cdot S_i) = \psi_i, $$
$$ \pi_{*} (\pi^{*}\eta \cdot S_i) = \eta. $$ 
Intersecting both sides of \eqref{eq:tauto} with $S_i$ and applying $\pi_{*}$, we obtain that 
$$ \eta =  (m_i + k) \psi_i, $$
thus verifying (i). For (ii), we have 
\begin{eqnarray*}
\kappa_{j-1} & = & \pi_{*} (\omega^{j}) \\ 
& = & \frac{1}{k^j} \sum_{h=0}^{j} {j\choose h} \eta^{j-h} \pi_{*}\left(\sum_{i=1}^n m_i S_i\right)^h \\
& = &   \frac{1}{k^j} \sum_{h=0}^{j} {j\choose h} \eta^{j-h} \left(  \sum_{i=1}^n m_i^h \pi_{*}(S_i^h)  \right) \\
& = &   \frac{1}{k^j} \sum_{h=0}^{j} {j\choose h} \eta^{j-h} \left(  \sum_{i=1}^n m_i^h  (-\psi_i)^{h-1} \right) \\
& = &   \frac{1}{k^j} \sum_{h=0}^{j} (-1)^{h-1}{j\choose h} \left(\sum_{i=1}^n  m_i^h (m_i + k)^{j-h}  \psi_i^{j-1} \right) \\
& = &   \frac{1}{k^j}\sum_{i=1}^n \psi_i^{j-1} \left(\sum_{h=0}^j   (-1)^{h-1} {j\choose h} m_i^h (m_i + k)^{j-h}  \right) \\
& = &   - \frac{1}{k^j}\sum_{i=1}^n \psi_i^{j-1} \left(\sum_{h=0}^j  {j\choose h} (-m_i)^h (m_i + k)^{j-h}  \right) \\
& = &   - \sum_{i=1}^n \psi_i^{j-1}. 
\end{eqnarray*} 
Hence the $\kappa$ classes can be generated by the $\psi$ classes, thus verifying (ii). 
\end{proof}

Now Theorem~\ref{thm:tauto} follows as a consequence of Proposition~\ref{prop:tauto}.

\begin{proof}[Proof of Theorem~\ref{thm:tauto}]
By definition $R^*(\PP^k(\mu))$ is generated by the $\psi$, $\kappa$ and $\eta$ classes. By Proposition~\ref{prop:tauto}, $R^*(\PP^k(\mu))$ is indeed generated by the $\psi$ and $\eta$ classes. If $m_i \neq -k$ for all $i$, then $\psi_i = \eta / (m_i + k)$ for all $i$, hence in this case $R^*(\PP^k(\mu))$ is generated by $\eta$. If some $m_i = -k$, then $\eta = 0$, hence $\psi_j = 0$ for all $j$ with $m_j \neq -k$. Therefore, 
in this case $R^*(\PP^k(\mu))$ is generated by those $\psi_i$ with $m_i = -k$. 
\end{proof}

We provide some examples to show that in general the generators of $R^*(\PP^k(\mu))$ in Theorem~\ref{thm:tauto} cannot be further reduced. 

\begin{example}
Consider $k = 1$ and $\mu = (1^{2g-2})$, i.e., the principal stratum of holomorphic abelian differentials with simple zeros only. Let $\PP$ be the projectivized Hodge bundle over $\MM_g$. Then for $g\geq 3$, the rational Picard group $\Pic_{\bbQ}(\PP)$ has rank two, generated by $\kappa_1$ and $\eta$. Denote by $\PP(\{1^{2g-2}\}) = \PP(1^{2g-2}) / S_{2g-2}$ where the zeros are unordered.   
Since $\PP(\{1^{2g-2}\})$ is the complement of the closure of the irreducible codimension-one stratum $\PP(2, \{1^{2g-4}\}) = \PP(2, 1^{2g-4}) / S_{2g-4}$ in $\PP$, it follows that $\Pic_{\bbQ}(\PP(\{1^{2g-2}\}))$ has rank one, generated by $\eta$ (or $\kappa_1$), hence in this case $\eta$ is nontrivial on $\PP(\{1^{2g-2}\})$ and $\PP(1^{2g-2})$. Indeed the divisor class of the closure of $\PP(2, \{1^{2g-4}\})$ in $\PP$ 
is equal to $2 \kappa_1 - (6g-6) \eta$ (see \cite{Cycle, KorotkinZografTau}).
\end{example}

\begin{example}
More generally, consider the stratum $\HH(m, 1^{2g-2-m})$ of holomorphic abelian differentials with a zero $p_1$ of order $m$ and $2g-2-m$ simple zeros. Denote by $\PP(m, \{1^{2g-2-m} \})$ the projectivized stratum where the $2g-2-m$ simple zeros are unordered. Let $\PP$ be the projectivized Hodge bundle of holomorphic canonical divisors 
over $\MM_{g,1}$. Marking the zero $p_1$ realizes $\PP(m, \{1^{2g-2-m} \})$ as a subvariety of $\PP$, and denote by $\PP(\overline{m, \{1^{2g-2-m} \}})$ its closure in $\PP$. The fiber of $\PP(\overline{m, \{1^{2g-2-m} \}})$ over $(C, p_1) \in \MM_{g,1}$ can be identified with the linear system $| K_C(-mp_1)|$, whose dimension is equal to 
$g-2-m + h^0(C, mp_1)$. It is easy to see that the locus of $(C, p_1)$ in $\MM_{g,1}$ with 
$h^0(C, mp_1) \geq n$ for some integer $n\geq 2$ has dimension bounded above by $2g + m - n -1$. Therefore, if $$(2g + m - n - 1) + (g-2-m + n) \leq \dim \PP(\overline{m, \{1^{2g-2-m} \}}) - 2 = 4g-5-m, $$
i.e., if $m \leq g-2$, then $\PP(\overline{m, \{1^{2g-2-m} \}})$ is isomorphic in codimension-one to a projective bundle over $\MM_{g,1}$ minus a higher codimension subvariety. It follows that $\Pic_{\bbQ}(\PP(\overline{m, \{1^{2g-2-m} \}}))$ has rank three, generated by $\kappa_1$, $\psi_1$ and $\eta$ when $ m\leq g-2$ and $g\geq 3$. Moreover, 
the complement of $\PP(m, \{1^{2g-2-m} \})$ in $\PP(\overline{m, \{1^{2g-2-m} \}})$ consists of the closures of two irreducible divisorial strata $\PP(m+1, \{1^{2g-3-m}\})$ and $\PP(m, 2, \{1^{2g-4-m} \})$ where in the former the zero order of $p_1$ becomes $m+1$ and in the latter two simple zeros merge. It implies that $\Pic_{\bbQ}(\PP(m, \{1^{2g-2-m} \}))$ has rank at least one. Since in this case the generators $\kappa_1$ and $\psi_1$ are expressible in terms of $\eta$, we thus conclude that $\Pic_{\bbQ}(\PP(m, \{1^{2g-2-m} \}))$ has rank exactly one, generated by $\eta$. Therefore, $\eta$ is nontrivial on $\PP(m, \{1^{2g-2-m} \})$ and $\PP(m, 1^{2g-2-m})$ when $1\leq m\leq g-2$ and $g\geq 3$. 
\end{example}

\begin{example}
Consider the stratum $\HH(-1^2, 1^{2g})$ of abelian differentials with two simple poles and $2g$ simple zeros. Denote by $p_1$ and $p_2$ the simple poles. Let $\HH(-1^2)$ be the (twisted) Hodge bundle over $\MM_{g,2}$ parameterizing 
$(C, p_1, p_2, \omega)$ such that $\omega$ is a holomorphic section of $K_C(p_1 + p_2)$, and let $\PP(-1^2)$ be its projectivization. 
For $g\geq 3$, $\Pic_{\bbQ}(\MM_{g,2})$ has a basis given by $\kappa_1$, $\psi_1$ and $\psi_2$. Then $\Pic_{\bbQ}(\PP(-1^2))$ has a basis given by $\kappa_1$, $\psi_1$, $\psi_2$ and $\eta$. Denote by $\PP(-1^2, \{1^{2g}\}) = \PP(-1^2, 1^{2g})/S_{2g}$ where the zeros are unordered. The complement of  $\PP(-1^2, \{1^{2g}\})$ 
in $\PP(-1^2)$ consists of the closures of two irreducible divisorial strata $\PP(-1^2, 2, \{1^{2g-2}\})$ and $\PP(0^2, \{1^{2g-2}\})$, where in the former 
$\omega$ specializes to having a non-simple zero and in the latter $\omega$ specializes to being a holomorphic section of $K_C$. It follows that $\Pic_{\bbQ}(\PP(-1^2, \{1^{2g}\}))$ has rank at least two. Since in this case $\kappa_1$ and $\eta$ are expressible in terms of $\psi_1$ and $\psi_2$, it implies that $\Pic_{\bbQ}(\PP(-1^2, \{1^{2g}\}))$ has 
a basis given by $\psi_1$ and $\psi_2$, hence $\psi_1$ and $\psi_2$ are also linearly independent on $\PP(-1^2, 1^{2g})$. 
\end{example}

It remains interesting to figure out which powers of $\eta$ are zero if all $m_i \neq -k$ as well as the relations between the products of the $\psi_i$ classes if some $m_i = -k$. We remark that the answer depends on each individual stratum (or its connected components). 

\begin{example}
For the principal stratum of abelian differentials we have seen that $\eta$ is nontrivial for $g\geq 3$. On the other hand, the hyperelliptic strata components $\PP(2g-2)^{\hyp}$ and $\PP(g-1,g-1)^{\hyp}$ are isomorphic to certain moduli spaces of pointed smooth rational curves, hence their rational Picard groups are trivial, and so is $\eta$. Similarly some low genus strata can be realized as images of finite morphisms from affine varieties with trivial rational Picard group (see \cite[Section 4]{Affine}), hence $\eta$ is trivial for those strata. 
\end{example}

\begin{example}
Consider the stratum $\PP(2, \{1^{2g-4} \})$ and its closure $\PP(\overline{2, \{1^{2g-4}\}})$ in the projectivized Hodge bundle $\PP$ over $\MM_{g}$. We have seen previously that $\Pic_{\bbQ}(\PP(2, \{1^{2g-4} \}))$ has rank one when $g \geq 4$. The complement of $\PP(2, \{1^{2g-4} \})$ in $\PP(\overline{2, \{1^{2g-4}\}})$ consists of the closures of two irreducible codimension-one strata $\PP(3, \{1^{2g-5} \})$ and $\PP(\{ 2, 2\}, \{1^{2g-4} \})$. Hence for $g\geq 4$, the rational Chow group 
$A^1(\PP(\overline{2, \{1^{2g-4}\}}))$ has rank at most three. 
Since the tautological part $R^2(\MM_g)\subset A^2(\MM_g)$ has a basis given by $\kappa_1^2$ and $\kappa_2$ for $g\geq 6$ (see \cite{Edidin}), in this range the classes $\kappa_1^2$, $\kappa_2$, $\kappa_1 \eta$ and $\eta^2$ are independent in $A^2(\PP)$. By the exact sequence 
$$ A^1(\PP(\overline{2, \{1^{2g-4}\}})) \to A^2(\PP) \to A^2(\PP(\{1^{2g-2}\})) \to 0, $$
we conclude that $A^2(\PP(\{1^{2g-2}\}))$ contains a nontrivial tautological class. Since $R^2(\PP(\{1^{2g-2}\}))$ is generated by $\eta^2$, it follows that $\eta^2 \neq 0$ for $\PP(\{1^{2g-2}\})$ and $\PP(1^{2g-2})$ when $g\geq 6$. 
\end{example}

\begin{remark}
In general if $m_i \neq -k$ for all $i$, by Proposition~\ref{prop:tauto} all the $\psi_i$ classes are multiples of $\eta$, and hence all the $\kappa_j$ classes are multiples of $\eta^j$. As $\kappa_{j} = 0$ for $j \geq g-1$ on $\MM_g$ (see \cite{Looijenga}), we conclude that $\eta^{g-1} = 0$ in this case. Similarly if some $m_i = -k$, then $\psi_{i_1}^{d_1}\cdots \psi_{i_\ell}^{d_\ell} = 0$ with $m_{i_j} = -k$ and $\sum_{j=1}^\ell d_{j} \geq g$, as it holds on $\MM_{g,n}$ (see \cite{GraberVakil}). 
\end{remark}



\begin{thebibliography}{Ab123}


\bibitem[BCGGM1]{BCGGM1}
M. Bainbridge, D. Chen, Q. Gendron, S. Grushevsky, and M. M\"oller, Compactification of strata of abelian differentials, arXiv:1604.08834. 

\bibitem[BCGGM2]{BCGGM2}
M. Bainbridge, D. Chen, Q. Gendron, S. Grushevsky, and M. M\"oller, Strata of $k$-differentials, arXiv:1610.09238. 

\bibitem[B]{Boissy}
C. Boissy, Connected components of the moduli space of meromorphic differentials,  
\emph{Comm. Math. Helv.} {\bf 90} (2015), no. 2, 255--286.   

\bibitem[C1]{Cycle}
D. Chen, Strata of abelian differentials and the Teichm\"uller dynamics, 
\emph{J. Mod. Dyn.} {\bf 7} (2013), no. 1, 135--152. 

\bibitem[C2]{Bootcamp}
D. Chen, Teichm\"uller dynamics in the eyes of an algebraic geometer,  
\emph{Proc. Sympos. Pure Math.} {\bf 95} (2017), 171--197.

\bibitem[C3]{Affine}
D. Chen, Affine geometry of strata of differentials, arXiv:1706.01142. 

\bibitem[E]{Edidin}
D. Edidin, The codimension-two homology of the moduli space of stable curves is algebraic, 
\emph{Duke Math. J.} {\bf 67} (1992), no. 3, 241--272. 

\bibitem[EKZ]{EKZ} 
 A. Eskin, M. Kontsevich, and A. Zorich, Sum of Lyapunov exponents of the Hodge bundle with respect to the Teichm\"uller geodesic flow, 
 \emph{Publ. Math. Inst. Hautes \'{E}tudes Sci.} {\bf 120} (2014), 207--333. 

\bibitem[EM]{EskinMirzakhani}
A. Eskin and M. Mirzakhani, Invariant and stationary measures for the $\SL(2,\RR)$ action on Moduli space, arXiv:1302.3320. 

\bibitem[EMM]{EskinMirzakhaniMohammadi}
A. Eskin, M. Mirzakhani, and A. Mohammadi, Isolation, equidistribution, and orbit closures for the $\SL(2,\RR)$ action on moduli space, 
\emph{Ann. of Math. (2)} {\bf 182} (2015), no. 2, 673--721. 

\bibitem[Fa]{Faber}
C. Faber, A conjectural description of the tautological ring of the moduli space of curves, 
\emph{Aspects Math.} {\bf E33} (1999), 109--129. 

\bibitem[FabP]{FaberPandharipande}
C. Faber and R. Pandharipande, Logarithmic series and Hodge integrals in the tautological ring, 
\emph{Michigan Math. J.} {\bf 48} (2000), 215--252. 

\bibitem[FarP]{FarkasPandharipande}
G. Farkas and R. Pandharipande, The moduli space of twisted canonical divisors, with an appendix by 
F. Janda, R. Pandharipande, A. Pixton, and D. Zvonkine, \emph{J. Inst. Math. Jussieu}, to appear. 

\bibitem[Fi]{Filip} 
S. Filip, Splitting mixed Hodge structures over affine invariant manifolds, 
\emph{Ann. of Math. (2)} {\bf 183} (2016), 681--713. 

\bibitem[GV]{GraberVakil}
T. Graber and R. Vakil, Relative virtual localization and vanishing of tautological classes on moduli spaces of curves, 
\emph{Duke Math. J.} {\bf 130} (2005), no. 1, 1--37. 

\bibitem[KonZor]{KontsevichZorich}
M. Kontsevich and A. Zorich, Connected components of the moduli spaces of Abelian differentials with prescribed singularities, 
{\em Invent. Math.} {\bf 153} (2003), no. 3, 631--678. 

\bibitem[KorZog]{KorotkinZografTau}
D. Korotkin and P. Zograf, Tau function and moduli of differentials, 
{\em Math. Res. Lett.} {\bf 18} (2011), no. 3, 447--458. 

\bibitem[La]{Lanneau}
E. Lanneau, Connected components of the strata of the moduli spaces of quadratic differentials, 
\emph{Ann. Sci. \'{E}c. Norm. Sup\'{e}r. (4)} {\bf 41} (2008), no. 1, 1--56. 

\bibitem[Lo]{Looijenga}
E. Looijenga, On the tautological ring of $\MM_g$, 
\emph{Invent. Math.} {\bf 121} (1995), no. 2, 411--419. 

\bibitem[S]{Sauvaget}
A. Sauvaget, Cohomology classes of strata of differentials, arXiv:1701.07867. 

\bibitem[W]{WrightLectures}
A. Wright, Translation surfaces and their orbit closures: an introduction for a broad audience, 
\emph{EMS Surv. Math. Sci.} {\bf 2} (2015), 63--108. 

\bibitem[Z]{ZorichFlat}
A. Zorich, Flat surfaces, 
\emph{Frontiers in number theory, physics, and geometry. I}, 437--583, Springer, Berlin, 2006. 


\end{thebibliography}
\end{document}